\documentclass[12pt]{amsart}

\usepackage[T1]{fontenc}
\usepackage{lmodern}
\usepackage{amsmath,amssymb,amsthm,mathtools}
\usepackage[a4paper,margin=30mm]{geometry}
\usepackage{microtype}
\usepackage[hidelinks]{hyperref}
\usepackage{enumitem}
\usepackage[numbers,sort]{natbib}
\bibpunct{[}{]}{;}{n}{}{}

\newcommand{\T}{\mathbb T^d}
\newcommand{\Z}{\mathbb Z^d}
\newcommand{\Zp}{\mathbb N_0^d}
\newcommand{\D}{\mathbb D^d}
\newcommand{\cD}{\mathcal D}
\newcommand{\cF}{\mathcal F}
\newcommand{\cB}{\mathcal B}
\newcommand{\cX}{\mathcal X}
\newcommand{\Op}{\operatorname{Op}}
\newcommand{\supp}{\operatorname{supp}}
\newcommand{\dd}{\,\mathrm d}
\newcommand{\ip}[2]{#1\mathbin{\cdot}#2}
\newcommand{\tdist}{\operatorname{dist}_{\mathbb T^d}}
\newcommand{\la}{\langle}
\newcommand{\ra}{\rangle}
\newcommand{\wh}{\widehat}
\newcommand{\norm}[1]{\left\lVert#1\right\rVert}

\theoremstyle{plain}
\newtheorem{theorem}{Theorem}[section]
\newtheorem{proposition}[theorem]{Proposition}
\newtheorem{lemma}[theorem]{Lemma}
\newtheorem{corollary}[theorem]{Corollary}
\theoremstyle{definition}
\newtheorem{definition}[theorem]{Definition}
\theoremstyle{remark}
\newtheorem{remark}[theorem]{Remark}

\title{Holomorphic Toroidal Pseudodifferential Operators on the Polydisk}
\author[M. Nielsen]{Morten Nielsen}
\thanks{This work was supported by the Independent Research Fund Denmark, grant no.\ 5281-00046B}
\address{Department of Mathematical Sciences\\
  Aalborg University\\
  Thomas Manns Vej 23\\
  DK--9220 Aalborg\\
  Denmark}
\email{mnielsen@math.aau.dk}

\subjclass[2020]{Primary 47G30; Secondary 42B35, 46E35, 32A37}
\keywords{Toroidal pseudodifferential operator, triangular symbol,
  $(\rho,\delta)$ symbol class, Besov space, Triebel--Lizorkin space,
  holomorphic function space, Ruzhansky--Turunen calculus,
  positive-frequency distribution, polydisk}

\begin{document}

\begin{abstract}
We give a necessary and sufficient triangular condition characterizing the
toroidal pseudodifferential operators on $\mathbb T^d$ that preserve the
positive-frequency cone $\mathbb N_0^d$ and hence act on holomorphic boundary
values on the polydisk.  For symbols of type $(1,0)$, we prove boundedness on
holomorphic Besov and Triebel--Lizorkin spaces throughout the quasi-Banach
range.  For $S^m_{\rho,\delta}$, $0\leq\delta<\rho\leq1$, we obtain the
critical loss $d(1-\rho)|1/p-1/2|$, together with endpoint estimates.  A
positive-cone oscillatory multiplier proves sharpness of the loss and
necessity of the Besov endpoint condition $q\leq t$.  As an application, we prove
well-posedness for a first-order holomorphic differential operator on these
scales.
\end{abstract}

\maketitle

\section{Introduction}

Ruzhansky and Turunen developed a systematic theory of pseudodifferential
operators on the torus \cite{RT2010,RTbook}.  They showed that the
Kohn--Nirenberg quantization on $\mathbb R^d$ has a discrete analogue in
which the frequency variable ranges over $\mathbb Z^d$, and they developed
a symbolic calculus parallel to the Euclidean theory, including
composition, parametrices, and Sobolev boundedness.  See also Georgiadis and
Nielsen \cite{GN2016} for Euclidean mixed-norm Besov and Triebel--Lizorkin
estimates.  The toroidal calculus provides a convenient setting for
questions that are awkward to formulate on $\mathbb R^d$, since the
discrete frequency variable itself carries global information.
Distributions on $\T$ supported in the positive cone $\Zp$ are precisely
the distributional boundary values, on the distinguished boundary, of
holomorphic functions on the polydisk $\D=\{z\in\mathbb{C}^d:|z_j|<1\}$.  For background on boundary
function theory in polydisks,
see Rudin \cite{Rudin1969}.

This identification links the toroidal calculus to function-theoretic
operator theory on $\D$. Besov and Triebel--Lizorkin spaces of holomorphic
functions are defined here by restricting the periodic scales
$B^s_{p,q}(\T)$ and $F^s_{p,q}(\T)$ of Peetre \cite{Peetre1976} and
Triebel \cite{Triebel1983} to the positive cone.  When $d=1$, these are the
boundary-value spaces studied by Kyriazis and Petrushev \cite{KP2014} and
by Cleanthous et al.\ \cite{CGN2019}; we adopt the analogous boundary-value definition on the polydisk. A toroidal pseudodifferential
operator that happens to preserve positive-cone distributions therefore descends, through
this boundary identification, to an operator on the holomorphic scale,
provided one already has the corresponding boundedness statement on the
full, two-sided periodic scale. 
Boundedness on the full two-sided periodic scales, obtained through toroidal calculus and Littlewood--Paley methods of
Georgiadis--Nielsen type \cite{RT2010,GN2016}, together with Euclidean
$(\rho,\delta)$ analogues such as those of Park \cite{Park2018},
supply exactly this ingredient; what is missing is a description of which
toroidal symbols induce operators compatible with the holomorphic
restriction in the first place.

The multidimensional theory has more structure than merely a coordinatewise repetition of
the one-dimensional argument.  The positive orthant $\mathbb N_0^d$ is a
distinguished frequency cone, and the triangular condition introduced below
is precisely the symbolic condition for its invariance.  This formulation
applies in every dimension and yields operators on holomorphic function
spaces of several complex variables.  It also makes explicit the
dimensional dependence of the rough-symbol theory: for $0<\rho<1$, the
critical loss inherited from the Euclidean theory is
\[
 d(1-\rho)\left|\frac1p-\frac12\right|.
\]
Thus, cone invariance and the dimension-dependent $(\rho,\delta)$ bounds are
distinct structural features of the multidimensional theory.

Fix $d\in\mathbb N$ and let
$\mathbb T^d=\mathbb R^d/(2\pi\mathbb Z)^d$, with normalized Haar measure
$\dd x/(2\pi)^d$.  For $f\in L^1(\T)$, its Fourier coefficients are
\begin{equation}\label{eq:fourier-coefficients}
 \wh f(n):=\frac1{(2\pi)^d}\int_{[0,2\pi]^d}
 f(x)e^{-i\ip n x}\dd x,\qquad n\in\Z,
\end{equation}
and the definition extends in the usual way to $f\in\cD'(\T)$, the space of
periodic distributions, by duality with trigonometric polynomials.  The
normalizing factor in Haar integrals over $\T$ is occasionally suppressed
in estimates below, since it only changes fixed constants.  The
distributional boundary value of a holomorphic function on $\D$ has Fourier
spectrum in $\Zp$.  The appropriate distribution space is therefore
\[
 \cD'_+(\T):=\{f\in\cD'(\T):\wh f(n)=0\text{ for }n\notin\Zp\}.
\]
The toroidal quantization of a symbol $a(x,n)$ is
\begin{equation}\label{eq:quantization}
 A_a f(x)=a(X,D)f(x):=\sum_{n\in\Z}a(x,n)\wh f(n)e^{i\ip n x}.
\end{equation}
For a general $x$-dependent symbol, $A_a$ does not preserve $\cD'_+(\T)$.
One of the main aims of this article is to identify the exact
Fourier-triangular condition that guarantees this invariance and to prove
boundedness on the full scales of holomorphic Besov and Triebel--Lizorkin
spaces.

The symbol condition below is strictly weaker than requiring $a(\cdot,n)$ to
be holomorphic.  It allows negative components in the $x$-frequency,
provided that the resulting shift of an input frequency $n$ does not cross
a face of the positive cone.  In particular, the resulting symbol class
includes both forward and backward coordinate shifts.

The structure of the paper is as follows: Section~\ref{sec:preliminaries}
introduces the triangular symbol class and the periodic Besov and
Triebel--Lizorkin scales; Section~\ref{sec:boundedness} establishes the
dyadic almost-orthogonality and maximal-function estimates that drive the
ambient boundedness theorem; Section~\ref{sec:main} characterizes
positive-frequency invariance and combines it with ambient boundedness to
prove the main theorem on the holomorphic scales; and
Section~\ref{sec:extensions} extends the results to a sharp boundedness result for the general
$S^m_{\rho,\delta}$ classes. Finally, in Section \ref{sec:application} we 
apply the theory to a specific first-order
holomorphic differential equation.

\section{Preliminaries: symbols and function spaces}\label{sec:preliminaries}

This section fixes the notation used throughout the paper.  We first
introduce the triangular toroidal symbol classes, and then recall the periodic Besov and
Triebel--Lizorkin scales on which the associated operators act, see
Section~\ref{sec:spaces}.

\subsection{Toroidal symbols and the triangular condition}\label{sec:symbols}

For $g:\Z\to\mathbb C$ and the standard coordinate vectors $e_j$, let
\[
 \Delta_{n_j}g(n):=g(n+e_j)-g(n),\qquad
 \Delta_n^\alpha:=\Delta_{n_1}^{\alpha_1}\cdots
 \Delta_{n_d}^{\alpha_d}.
\]
The subscript distinguishes these discrete differences from the
Littlewood--Paley projections $\Delta_j$ used below.
As usual, the bracket is given by $\la n\ra=(1+|n|^2)^{1/2}$ and standard multiindex notation is used below. Let us recall the toroidal symbol class
$S^m_{1,0}(\T\times\Z)$ introduced by Ruzhansky and Turunen
\cite{RT2010,RTbook}.

\begin{definition}\label{def:symbol}
Let $m\in\mathbb R$.  The toroidal symbol class
$S^m_{1,0}(\T\times\Z)$ consists of the functions
$a:\T\times\Z\to\mathbb C$ which are smooth in $x$ and satisfy
\begin{equation}\label{eq:symbol-estimate}
 \sup_{x\in\T,n\in\Z}
 \la n\ra^{-m+|\alpha|}
 \left|\Delta_n^\alpha\partial_x^\beta a(x,n)\right|<\infty
\end{equation}
for every $\alpha,\beta\in\mathbb N_0^d$.
\end{definition}

We may expand the symbol in the first variable:
\begin{equation}\label{eq:x-fourier}
 a(x,n)=\sum_{\nu\in\Z}a_\nu(n)e^{i\ip\nu x},\qquad
 a_\nu(n)=\frac1{(2\pi)^d}\int_{[0,2\pi]^d}
 a(x,n)e^{-i\ip\nu x}\dd x.
\end{equation}
We also record that the distributional kernel of $A_a$ is
\[
 K_a(x,y)=\sum_{n\in\Z}e^{i\ip n{(x-y)}}a(x,n).
\]

The class of positive triangular symbols will be central in this study.

\begin{definition}\label{def:triangular}
A symbol $a\in S^m_{1,0}(\T\times\Z)$ is called
\emph{positive triangular} if
\begin{equation}\label{eq:triangular}
 a_\nu(n)=0
 \quad\text{whenever}\quad n\in\Zp\quad\text{and}\quad n+\nu\notin\Zp.
\end{equation}
We write $S^m_{1,0,\triangle}$ for this class.
\end{definition}

Condition \eqref{eq:triangular} is necessary as well as sufficient in a suitable sense; see
Proposition~\ref{prop:invariance} and the discussion  in Remark \ref{rem:sym} below.  The following lemma gives the
fundamental decay estimate for the Fourier coefficients of symbols in
$S^m_{1,0}(\T\times\Z)$.

\begin{lemma}\label{lem:coefficients}
If $a\in S^m_{1,0}(\T\times\Z)$, then for every
$\alpha\in\mathbb N_0^d$ and $L\in\mathbb N_0$ there is $C_{\alpha,L}$ such that
\begin{equation}\label{eq:coefficient-estimate}
 |\Delta_n^\alpha a_\nu(n)|
 \leq C_{\alpha,L}(1+|\nu|)^{-L}\la n\ra^{m-|\alpha|}.
\end{equation}
\end{lemma}

\begin{proof}
The difference operators commute with the Fourier coefficient integral.  For
$\nu\ne0$, choose $j$ with $|\nu_j|\geq|\nu|/\sqrt d$ and integrate by
parts $L$ times in $x_j$ to obtain
\[
 \Delta_n^\alpha a_\nu(n)
 =(i\nu_j)^{-L}\frac1{(2\pi)^d}\int_{[0,2\pi]^d}
 \partial_{x_j}^L\Delta_n^\alpha a(x,n)e^{-i\ip\nu x}\dd x.
\]
Estimate \eqref{eq:symbol-estimate} proves the assertion for $\nu\ne0$.
The case $\nu=0$ follows directly from \eqref{eq:symbol-estimate}.
\end{proof}

With the symbol classes fixed, we turn to the function spaces on which the
associated operators act.

\subsection{Periodic Besov and Triebel--Lizorkin spaces}\label{sec:spaces}

Choose $\varphi_0,\varphi\in C_c^\infty(\mathbb R^d)$ such that
\[
 \supp\varphi_0\subset\{|\xi|\leq2\},\qquad
 \supp\varphi\subset\{1/2\leq|\xi|\leq2\},
\]
and
\[
 \varphi_0(\xi)+\sum_{j=1}^\infty\varphi(2^{-j}\xi)=1,
 \qquad \xi\in\mathbb R^d.
\]
Set $\varphi_j(\xi)=\varphi(2^{-j}\xi)$ for $j\geq1$ and
\[
 \Delta_jf:=\varphi_j(D)f
   =\sum_{n\in\Z}\varphi_j(n)\wh f(n)e^{i\ip n x}.
\]

For $s\in\mathbb R$ and $0<p,q\leq\infty$, define
\begin{equation}\label{eq:besov-norm}
 \norm{f}_{B^s_{p,q}(\T)}
 :=\norm{\{2^{js}\norm{\Delta_jf}_{L^p(\T)}\}_{j\geq0}}_{\ell^q}.
\end{equation}
For $s\in\mathbb R$, $0<p<\infty$, and $0<q\leq\infty$, define
\begin{equation}\label{eq:f-norm}
 \norm{f}_{F^s_{p,q}(\T)}
 :=\norm{\{2^{js}\Delta_jf\}_{j\geq0}}_{L^p(\T;\ell^q)}.
\end{equation}
Here $L^p(\T;\ell^q)$ denotes the mixed-norm space of sequences
$\{g_j(x)\}_{j\geq0}$ of functions on $\T$ with finite quasi-norm
$\norm{\{g_j\}_j}_{L^p(\T;\ell^q)}:=\norm{\,\norm{\{g_j(x)\}_j}_{\ell^q}\,}_{L^p_x(\T)}$,
i.e., the $\ell^q$ norm is taken pointwise in $x$ before the $L^p$ norm is
taken in $x$.  The usual changes are made when an exponent equals infinity.  Different
admissible resolutions give equivalent quasi-norms.  These are the periodic
counterparts of the classical (quasi-)Banach scales of Peetre
\cite{Peetre1976} and Triebel \cite{Triebel1983}.

The following positive-frequency subspaces
\begin{align*}
 B^s_{p,q,+}(\T)&:=B^s_{p,q}(\T)\cap\cD'_+(\T),\\
 F^s_{p,q,+}(\T)&:=F^s_{p,q}(\T)\cap\cD'_+(\T)
\end{align*}
will be central to this study.
\section{Boundedness on the periodic scales}\label{sec:boundedness}

Boundedness of $A_a$ on the periodic Besov and Triebel--Lizorkin scales is
proved by a dyadic almost-orthogonality argument.  We first establish
pointwise kernel and maximal-function estimates, and then assemble them into the ambient boundedness
theorem in Section~\ref{sec:ambient}.

\subsection{Dyadic estimates for toroidal operators}\label{sec:dyadic}

We give a self-contained version of the almost-orthogonality estimate needed
below.  Let $\tdist(x,y)$ denote the geodesic distance on $\T$.
For $0<r<\infty$, write
\[
 \mathcal M_rg(x):=\bigl(M(|g|^r)(x)\bigr)^{1/r},
\]
where $M$ is the Hardy--Littlewood maximal operator over balls in $\T$.

Choose $\psi_0,\psi\in C_c^\infty(\mathbb R^d)$ so that
$\psi_j=1$ on a neighbourhood of $\supp\varphi_j$, with the same type of
dyadic support, and put $\Psi_j=\psi_j(D)$.  Then
$f=\sum_{k\geq0}\Psi_k\Delta_kf$, with harmless finite overlaps.  For
notational economy below we apply $A_a$ to an arbitrary function $g_k$ whose
Fourier transform is supported in $|n|\leq C2^k$ and, for $k\geq1$, in an
annulus $c2^k\leq|n|\leq C2^k$.

\begin{lemma}\label{lem:kernel}
Let $a\in S^m_{1,0}(\T\times\Z)$.  For every $N,L>0$ there is a constant
$C_{N,L}$ such that the kernel $K_{j,k}$ of
$\Delta_jA_a\Psi_k$ satisfies
\begin{equation}\label{eq:kernel-estimate}
 |K_{j,k}(x,y)|
 \leq C_{N,L}\,2^{km}2^{-N|j-k|}
 2^{kd}(1+2^k\tdist(x,y))^{-L}
\end{equation}
for $j,k\geq0$.  The factors $2^{km}$ and $2^{kd}$ are interpreted as bounded
constants when $k=0$.
\end{lemma}

\begin{proof}
We first suppress the output projection.  The kernel of $A_a\Psi_k$ is
\begin{equation}\label{eq:basic-kernel}
 H_k(x,y)=\sum_{n\in\Z}a(x,n)\psi_k(n)e^{i\ip n{(x-y)}}.
\end{equation}
On the support of $\psi_k$, the symbol estimates and the discrete Leibniz
formula imply, for every multiindex $\gamma$,
\begin{equation}\label{eq:difference-product}
 \left|\Delta_n^\gamma\{a(x,n)\psi_k(n)\}\right|
 \leq C_\gamma2^{k(m-|\gamma|)}.
\end{equation}
Indeed, every term in the discrete Leibniz formula contains a difference of
multiorder $\mu$ of $a$, bounded by $C2^{k(m-|\mu|)}$, and a difference of
multiorder $\gamma-\mu$ of $\psi_k$, bounded by
$C2^{-k(|\gamma|-|\mu|)}$.  The shifts in the discrete Leibniz
formula do not change these estimates.

For $t=x-y$, summation by parts gives
\[
 \prod_{j=1}^d(e^{-it_j}-1)^{\gamma_j}H_k(x,y)
 =\sum_n\Delta_n^\gamma\{a(x,n)\psi_k(n)\}e^{i\ip n t},
\]
up to irrelevant signs and integer shifts.  Applying this identity in a
coordinate direction in which $|t_j|\geq |t|/\sqrt d$, and iterating, gives
arbitrary decay in $2^k|t|$.  Since the sum contains $O(2^{kd})$ terms, the
trivial estimate is $|H_k(x,y)|\leq C2^{k(m+d)}$, and hence
\begin{equation}\label{eq:basic-localization}
 |H_k(x,y)|\leq C_L2^{k(m+d)}
 (1+2^k\tdist(x,y))^{-L}.
\end{equation}
The same argument after applying any number of $x$-derivatives gives
\begin{equation}\label{eq:xder-kernel}
 |\partial_x^\gamma H_k(x,y)|
 \leq C_{L,\gamma}2^{k(m+|\gamma|+d)}
 (1+2^k\tdist(x,y))^{-L}.
\end{equation}
Here differentiating $e^{i\ip n{(x-y)}}$ costs $O(2^k)$, while differentiating the
symbol is controlled by \eqref{eq:symbol-estimate}.

We now insert the output projection.  If $|j-k|\leq3$, convolution of
\eqref{eq:basic-localization} with the rapidly decreasing periodic kernel of
$\Delta_j$ gives \eqref{eq:kernel-estimate} without the off-diagonal factor.

For both off-diagonal cases we use the Fourier expansion
\eqref{eq:x-fourier} in \eqref{eq:basic-kernel}.  The output projection forces
$\ell=n+\nu$ to satisfy $|\ell|\asymp2^j$, whereas $|n|\asymp2^k$.
If $j\geq k+4$, then $|\nu|\geq c2^j$; if $k\geq j+4$, then
$|\nu|\geq c2^k$.  Lemma~\ref{lem:coefficients}, with arbitrarily large
$L_0$, therefore supplies respectively the factors $2^{-jL_0}$ and
$2^{-kL_0}$.  Now we apply the discrete Leibniz formula and summation by parts in
$n$ to obtain
\[
 \sum_{n,\nu}\varphi_j(n+\nu)a_\nu(n)\psi_k(n)
 e^{i\ip{(n+\nu)}x-i\ip n y},
\]
noting that each difference of $a_\nu(n)$ is covered by
Lemma~\ref{lem:coefficients}, while differences of the cutoffs cost
$2^{-j}$ or $2^{-k}$.  This gives, for arbitrary $M,L>0$,
\[
 |K_{j,k}(x,y)|
 \leq C2^{km}2^{-M\max(j,k)}
 2^{kd}(1+2^k\tdist(x,y))^{-L}.
\]
Since $\max(j,k)\geq|j-k|$, this is actually stronger than the required
$2^{-N|j-k|}$ after choosing $M>N$.  The finitely many cases involving
$j=0$ or $k=0$ are absorbed by changing the constant.  This completes the
proof.
\end{proof}

To turn this kernel bound into an $L^p$ estimate we derive a suitable estimate for
integration against such a kernel by a maximal function, which is the
content of the next lemma.

\begin{lemma}\label{lem:max-domination}
Suppose
\[
 |K_k(x,y)|\leq C2^{kd}(1+2^k\tdist(x,y))^{-L}
\]
where $0<r\leq1$ and $L>d/r$.  Suppose also that $\wh g$ is supported in
$\{|n|\leq C_02^k\}$.  Then
\[
 \int_{\T}|K_k(x,y)g(y)|\,\dd y
 \leq C_{L,r,C_0}\mathcal M_rg(x).
\]
\end{lemma}

\begin{proof}
We first record the periodic Peetre--Fefferman--Stein estimate
\begin{equation}\label{eq:submean}
 |g(y)|
 \leq C_{R,r}
 \left(\int_{\T} 2^{kd}(1+2^k\tdist(y,z))^{-R}|g(z)|^r\dd z\right)^{1/r}.
\end{equation}
To prove it, choose a smooth multiplier $\eta_k(D)$ which equals one on the
spectrum of $g$.  Its kernel satisfies, for every $A>0$,
\[
 |E_k(u)|\leq C_A2^{kd}(1+2^k\tdist(u,0))^{-A},
\]
and $g=E_k*g$.    Fix \(N>0\) and define the Peetre maximal function
\[
 P_{k,N}g(y)
 :=
 \sup_{w\in\mathbb T^d}
 \frac{|g(w)|}
      {\bigl(1+2^k\operatorname{dist}_{\mathbb T^d}(y,w)\bigr)^N}.
\]
This quantity is pointwise finite because \(g\) is a trigonometric polynomial. From
the reproducing formula \(g=E_k*g\), we have
\[
 |g(w)|
 \leq
 \int_{\mathbb T^d}
 |E_k(w-z)|\,|g(z)|^r|g(z)|^{1-r}\,dz.
\]
Moreover,
\[
 |g(z)|^{1-r}
 \leq
 P_{k,N}g(y)^{1-r}
 \bigl(1+2^k\operatorname{dist}_{\mathbb T^d}(y,z)\bigr)^{N(1-r)}.
\]
Using
\[
 1+2^k\operatorname{dist}_{\mathbb T^d}(y,z)
 \leq
 \bigl(1+2^k\operatorname{dist}_{\mathbb T^d}(y,w)\bigr)
 \bigl(1+2^k\operatorname{dist}_{\mathbb T^d}(w,z)\bigr)
\]
and choosing the decay order \(A\) of \(E_k\) sufficiently large, we obtain
\begin{align*}
 \frac{|g(w)|}
      {\bigl(1+2^k\operatorname{dist}_{\mathbb T^d}(y,w)\bigr)^N}
 &\leq
 C\,P_{k,N}g(y)^{1-r}
 \int_{\mathbb T^d}
 2^{kd}
 \frac{|g(z)|^r}
      {\bigl(1+2^k\operatorname{dist}_{\mathbb T^d}(y,z)\bigr)^{Nr}}
 \,dz.
\end{align*}
Taking the supremum over \(w\in\mathbb T^d\) gives
\[
 P_{k,N}g(y)
 \leq
 C\,P_{k,N}g(y)^{1-r}
 \int_{\mathbb T^d}
 2^{kd}
 \frac{|g(z)|^r}
      {\bigl(1+2^k\operatorname{dist}_{\mathbb T^d}(y,z)\bigr)^{Nr}}
 \,dz.
\]
If \(P_{k,N}g(y)=0\), the desired estimate is immediate. Otherwise,
division by \(P_{k,N}g(y)^{1-r}\) yields
\[
 P_{k,N}g(y)^r
 \leq
 C
 \int_{\mathbb T^d}
 2^{kd}
 \frac{|g(z)|^r}
      {\bigl(1+2^k\operatorname{dist}_{\mathbb T^d}(y,z)\bigr)^{Nr}}
 \,dz.
\]
Since \(|g(y)|\leq P_{k,N}g(y)\), taking the \(r\)th root and writing
\(R=Nr\) proves \eqref{eq:submean}.

We next prove the  local integral estimate
\begin{equation}\label{eq:local-integral}
 \int_{\tdist(x,y)<2^{u-k}}|g(y)|\dd y
 \leq C2^{ud/r-kd}\mathcal M_rg(x),\qquad u\geq0.
\end{equation}
It suffices to consider those $u$ for which $2^{u-k}$ does not
exceed a fixed multiple of the diameter of $\mathbb T^d$; the remaining
cases follow after changing the constant.

Partition
\[
 B_u(x):=\bigl\{y\in\mathbb T^d:
              \operatorname{dist}_{\mathbb T^d}(x,y)<2^{u-k}\bigr\}
\]
into cubes $Q_v$ of side length comparable to $2^{-k}$, with uniformly
bounded overlap, and choose a point $y_v\in Q_v$. If $y\in Q_v$, then
the triangle inequality gives
\[
 1+2^k\operatorname{dist}_{\mathbb T^d}(y_v,z)
 \leq C\bigl(1+2^k\operatorname{dist}_{\mathbb T^d}(y,z)\bigr),
\]
and conversely. Hence \eqref{eq:submean}, with an exponent $R>d$, implies
\[
 \sup_{y\in Q_v}|g(y)|
 \leq C
 \left(
   2^{kd}\int_{\mathbb T^d}
   \frac{|g(z)|^r}
        {\bigl(1+2^k\operatorname{dist}_{\mathbb T^d}(y_v,z)\bigr)^R}
   \,dz
 \right)^{1/r}.
\]
Consequently,
\[
 \int_{Q_v}|g(y)|\,dy
 \leq C\,2^{-kd}
 \left(
   2^{kd}\int_{\mathbb T^d}
   \frac{|g(z)|^r}
        {\bigl(1+2^k\operatorname{dist}_{\mathbb T^d}(y_v,z)\bigr)^R}
   \,dz
 \right)^{1/r}.
\]
Since $0<r\leq 1$, the inequality
\[
 \left(\sum_v c_v\right)^r\leq\sum_v c_v^r
\]
therefore gives
\begin{align}
 \left(\int_{B_u(x)}|g(y)|\,dy\right)^r
 &\leq
 C\,2^{-kdr}2^{kd}
 \int_{\mathbb T^d}|g(z)|^r W_{u,k}(x,z)\,dz,
 \label{eq:local-integral-sum}
\end{align}
where
\[
 W_{u,k}(x,z)
 :=
 \sum_{v:\,Q_v\cap B_u(x)\neq\varnothing}
 \bigl(1+2^k\operatorname{dist}_{\mathbb T^d}(y_v,z)\bigr)^{-R}.
\]

Choose $R>d$ sufficiently large. A standard lattice-sum estimate gives,
for any prescribed $R_0>d$, after choosing $R>R_0+d$,
\begin{equation}\label{eq:lattice-sum-estimate}
 W_{u,k}(x,z)
 \leq
 C_{R_0}\,
 \bigl(1+2^{k-u}
       \operatorname{dist}_{\mathbb T^d}(x,z)\bigr)^{-R_0}.
\end{equation}
Indeed, when
$\operatorname{dist}_{\mathbb T^d}(x,z)\leq C2^{u-k}$, the sum is
bounded by the full lattice sum
$\sum_{\nu\in\mathbb Z^d}(1+|\nu|)^{-R}$.
Outside this enlarged ball, the distance from $z$ to every $y_v$ is
comparable to $\operatorname{dist}_{\mathbb T^d}(x,z)$, and summation
over the $O(2^{ud})$ cubes yields \eqref{eq:lattice-sum-estimate}, provided
$R$ is chosen sufficiently large.

Decomposing $\mathbb T^d$ into annuli centred at $x$, estimate
\eqref{eq:lattice-sum-estimate} and the definition of the Hardy--Littlewood
maximal function imply
\begin{align}
 \int_{\mathbb T^d}
 |g(z)|^r W_{u,k}(x,z)\,dz
 &\leq
 C\,2^{(u-k)d}M(|g|^r)(x).
 \label{eq:weighted-maximal-estimate}
\end{align}
Indeed, on the annulus
\[
 2^{h+u-k}\leq
 \operatorname{dist}_{\mathbb T^d}(x,z)
 <2^{h+1+u-k},
\]
the weight in \eqref{eq:lattice-sum-estimate} is $O(2^{-hR_0})$, whereas
the integral of
$|g|^r$ is at most
\[
 C\,2^{(h+u-k)d}M(|g|^r)(x).
\]
The resulting series converges due to the fact that $R_0>d$.

Combining \eqref{eq:local-integral-sum} and
\eqref{eq:weighted-maximal-estimate}, we obtain
\[
 \left(\int_{B_u(x)}|g(y)|\,dy\right)^r
 \leq
 C\,2^{-kdr}2^{kd}2^{(u-k)d}M(|g|^r)(x)
 =
 C\,2^{ud-kdr}M_rg(x)^r.
\]
Taking the $r$th root proves \eqref{eq:local-integral}.

Finally divide $\T$ into $E_0=\{\tdist(x,y)<2^{-k}\}$ and
\[
 E_h=\{2^{h-1-k}\leq\tdist(x,y)<2^{h-k}\},\qquad h\geq1,
\]
stopping when the outer radius exceeds the diameter of $\T$.  On $E_h$ the
kernel is at most $C2^{kd}2^{-hL}$.  Estimate
\eqref{eq:local-integral} gives
\[
 \int_{E_h}|g(y)|\dd y
 \leq C2^{hd/r-kd}\mathcal M_rg(x).
\]
Thus the
$h$th contribution is bounded by
$C2^{-h(L-d/r)}\mathcal M_rg(x)$.  Summation in $h$ then proves the claim.
\end{proof}

Combining Lemmas~\ref{lem:kernel} and \ref{lem:max-domination} gives the
pointwise almost-orthogonality estimate
\begin{equation}\label{eq:pointwise-AO}
 |\Delta_jA_a\Psi_kg(x)|
 \leq C_{N,r}2^{km}2^{-N|j-k|}\mathcal M_rg(x),
\end{equation}
provided the spatial decay exponent in Lemma~\ref{lem:kernel} is chosen
larger than $d/r$.

These pointwise estimates are exactly what is needed to apply a discrete
Littlewood--Paley argument, which we consider next.

\subsection{Ambient boundedness}\label{sec:ambient}
We have the following boundedness result.
\begin{proposition}\label{prop:ambient}
Let $a\in S^m_{1,0}(\T\times\Z)$ and $s\in\mathbb R$.
Then
\begin{equation}\label{eq:ambient-B}
 A_a:B^{s+m}_{p,q}(\T)\longrightarrow B^s_{p,q}(\T)
\end{equation}
for $0<p,q\leq\infty$, and
\begin{equation}\label{eq:ambient-F}
 A_a:F^{s+m}_{p,q}(\T)\longrightarrow F^s_{p,q}(\T)
\end{equation}
for $0<p<\infty$ and $0<q\leq\infty$.
\end{proposition}

\begin{proof}
Using a second resolution of unity, write
\[
 \Delta_jA_af=\sum_{k\geq0}\Delta_jA_a\Psi_k\Delta_kf.
\]
By \eqref{eq:pointwise-AO},
\begin{align}
 2^{js}|\Delta_jA_af(x)|
 &\leq C\sum_{k\geq0}2^{js}2^{km}2^{-N|j-k|}
       \mathcal M_r(\Delta_kf)(x)\notag\\
 &\leq C\sum_{k\geq0}2^{-(N-|s|)|j-k|}
       2^{k(s+m)}\mathcal M_r(\Delta_kf)(x).
 \label{eq:weighted-AO}
\end{align}
Choose $N>|s|+1$.

For the Besov estimate, choose $0<r<p$ when $p<\infty$.  Since $p/r>1$,
the Hardy--Littlewood maximal theorem gives
\[
 \norm{\mathcal M_rg}_{L^p}\leq C_{p,r}\norm g_{L^p}.
\]
For $p=\infty$ the same estimate is immediate.  Take the $L^p$ quasi-norm in
\eqref{eq:weighted-AO} and put
\[
 \eta:=\min(1,p,q),\qquad
 d_\ell:=2^{-\eta(N-|s|)|\ell|},\qquad
 c_k:=2^{k(s+m)}\norm{\Delta_kf}_{L^p}.
\]
The $p$-quasi-triangle inequality and
$(\sum c_k)^\eta\leq\sum c_k^\eta$ give
\[
 \bigl(2^{js}\norm{\Delta_jA_af}_{L^p}\bigr)^\eta
 \leq C\sum_k d_{j-k}c_k^\eta.
\]
Here the same formula covers $p\geq1$, after possibly decreasing $\eta$.
Since $q/\eta\geq1$ and $d\in\ell^1$, Young's inequality on
$\ell^{q/\eta}$ yields
\[
 \norm{\{(d*c^\eta)_j^{1/\eta}\}_j}_{\ell^q}
 \leq \norm d_{\ell^1}^{1/\eta}\norm c_{\ell^q}.
\]
This is the required powered discrete convolution estimate, and it yields
\[
 \norm{A_af}_{B^s_{p,q}}
 \leq C\norm f_{B^{s+m}_{p,q}}.
\]

For the Triebel--Lizorkin estimate first suppose $q<\infty$ and choose
$0<r<\min(p,q,1)$.  The vector-valued maximal inequality applied with
exponents $p/r>1$ and $q/r>1$ gives
\begin{equation}\label{eq:FS}
 \norm{\{\mathcal M_rg_k\}_k}_{L^p(\ell^q)}
 \leq C\norm{\{g_k\}_k}_{L^p(\ell^q)}.
\end{equation}
Applying the discrete convolution estimate in the $j$ index to
\eqref{eq:weighted-AO}, followed by \eqref{eq:FS}, gives
\[
 \norm{A_af}_{F^s_{p,q}}
 \leq C\norm f_{F^{s+m}_{p,q}}.
\]
When $q=\infty$, use the elementary estimate
\[
 \sup_k\mathcal M_rg_k
 \leq\mathcal M_r\left(\sup_k|g_k|\right)
\]
and choose $0<r<p$.  This proves \eqref{eq:ambient-F} in the remaining case.
\end{proof}

\begin{remark}
The proof uses only finitely many symbol seminorms for each fixed collection
of parameters.  More precisely, choose $0<r<\min(p,q,1)$ in the
Triebel--Lizorkin case and $0<r<p$ in the Besov case, followed by
$N>|s|+1$ and $L>d/r$.  Lemma~\ref{lem:kernel} then requires only finitely
many $x$-derivatives and $n$-differences, with their orders depending on
$N$ and $L$.  Tracking the orders used in the Taylor expansion and
summation by parts yields an explicit finite-regularity version of the
result.
\end{remark}

\section{From periodic to holomorphic spaces}\label{sec:main}

It remains to combine the ambient boundedness theorem with a description of
which symbols respect the positive-frequency subspace.  We characterize
this compatibility below and then prove the main holomorphic
theorem in Section~\ref{sec:holomorphic}.

\subsection{Positive-frequency invariance}\label{sec:invariance}

\begin{proposition}\label{prop:invariance}
Let $a\in S^m_{1,0}(\T\times\Z)$.  Then
\[
 A_a:\cD'_+(\T)\longrightarrow\cD'_+(\T)
\]
is continuous if and only if $a$ satisfies the triangular condition
\eqref{eq:triangular}.
\end{proposition}

\begin{proof}
For a trigonometric polynomial $f$ with spectrum in $\Zp$,
\begin{align}
 A_af(x)
 &=\sum_{n\in\Zp}\sum_{\nu\in\Z}
   a_\nu(n)\wh f(n)e^{i\ip{(n+\nu)}x},\label{eq:expanded-action}
\end{align}
and hence
\begin{equation}\label{eq:output-coefficient}
 \wh{A_af}(\ell)
 =\sum_{n\in\Zp}a_{\ell-n}(n)\wh f(n).
\end{equation}
If $\ell\notin\Zp$, then $n+(\ell-n)=\ell\notin\Zp$ for every
$n\in\Zp$.
The triangular condition therefore makes every term in
\eqref{eq:output-coefficient} vanish.

The calculation extends from polynomials to $\cD'_+$.  Indeed, for some
$M\geq0$,
\[
 |\wh f(n)|\leq C\la n\ra^M,\qquad n\in\Zp.
\]
Fix $\ell\notin\Zp$.  Since $|\ell-n|\asymp1+|n|$ as $|n|\to\infty$,
Lemma~\ref{lem:coefficients} gives, for every $L$,
\[
 |a_{\ell-n}(n)\wh f(n)|
 \leq C_{\ell,L}\la n\ra^{m+M-L}.
\]
For fixed $\ell$, choosing $L>m+M+d$ makes the series in
\eqref{eq:output-coefficient} absolutely convergent.  To control the output
uniformly in $\ell$, put $\gamma=m+M$ and use
$\la n\ra^\gamma\leq C\la\ell\ra^{\gamma_+}
(1+|\ell-n|)^{|\gamma|}$, where $\gamma_+=\max(\gamma,0)$.  Lemma~\ref{lem:coefficients}
then gives, for $L>|\gamma|+d$,
\[
 |\wh{A_af}(\ell)|
 \leq C\la\ell\ra^{\gamma_+}
 \sum_{n\in\Zp}(1+|\ell-n|)^{-L+|\gamma|}
 \leq C'\la\ell\ra^{\gamma_+}.
\]
The same estimate is uniform on bounded subsets of $\cD'_+(\T)$, for
which $M$ and the coefficient-growth constant may be chosen uniformly;
hence $A_a$ is continuous on $\cD'_+(\T)$.  Thus
\eqref{eq:output-coefficient} defines a distribution.  If
$\ell\notin\Zp$, every summand is zero by triangularity, so
$\wh{A_af}(\ell)=0$.

Conversely, suppose $A_a$ maps $\cD'_+$ continuously into itself.  Fix $n\in\Zp$ and apply
$A_a$ to $e^{i\ip n x}$.  Formula \eqref{eq:expanded-action} becomes
\[
 A_a(e^{i\ip n x})=\sum_{\nu\in\Z}a_\nu(n)e^{i\ip{(n+\nu)}x}.
\]
Preservation of positive spectrum forces $a_\nu(n)=0$ whenever
$n+\nu\notin\Zp$.
This is exactly \eqref{eq:triangular}.
\end{proof}

\begin{remark}\label{rem:sym}
The stronger condition $a_\nu(n)=0$ for every $\nu\notin\Zp$ and
$n\in\Zp$ says
that $a(\cdot,n)$ is the boundary value of a holomorphic function.  It is
sufficient but not necessary.  For example, when $d\geq2$, the symbol
\[
 a(x,n)=e^{i(x_1-x_2)}\frac{n_2}{\la n\ra}
\]
belongs to $S^0_{1,0}(\T\times\Z)$.  On the positive cone it vanishes when
$n_2=0$, and otherwise shifts the frequency by $e_1-e_2$; hence it satisfies
\eqref{eq:triangular}.  It is not analytic in $x$ and genuinely couples two
coordinates.  The analogous one-coordinate example is
$e^{-ix_1}n_1/\la n\ra$.
\end{remark}

Together with the ambient boundedness theorem of Section~\ref{sec:ambient}, this
characterization immediately yields boundedness on the holomorphic scales.

\subsection{The holomorphic mapping theorem}\label{sec:holomorphic}

Throughout this paper, $\cB^s_{p,q}(\D)$ and $\cF^s_{p,q}(\D)$ denote, by
definition, the holomorphic functions whose distributional boundary values
on the distinguished boundary belong to $B^s_{p,q,+}(\T)$ and
$F^s_{p,q,+}(\T)$, respectively.  Thus the boundary-value map gives the
canonical identifications (isometric for the fixed Littlewood--Paley
quasi-norms used here),
\begin{equation}\label{eq:boundary-identification}
 \cB^s_{p,q}(\D)\simeq B^s_{p,q,+}(\T),\qquad
 \cF^s_{p,q}(\D)\simeq F^s_{p,q,+}(\T).
\end{equation}
Conversely, every distribution in $\cD'_+(\T)$ has polynomially growing
Fourier coefficients, so its positive-frequency Fourier series converges
normally on compact subsets of $\D$ and defines a holomorphic function
with the prescribed distributional boundary value.
Indeed, if
$f(z)=\sum_{n\in\Zp}c_nz^n$, its boundary distribution has coefficients
$\wh f(n)=c_n$ for $n\in\Zp$ and zero coefficients otherwise, and the dyadic
expressions defining the two sides of \eqref{eq:boundary-identification} are
the same.  In dimension one these definitions agree with the established
boundary-value spaces discussed in the Introduction.

\begin{theorem}
\label{thm:main}
Let $a\in S^m_{1,0,\triangle}(\T\times\Z)$ and $s,m\in\mathbb R$.
Then:
\begin{enumerate}[label=\textup{(\roman*)}]
\item for $0<p,q\leq\infty$,
\[
 A_a:B^{s+m}_{p,q,+}(\T)\longrightarrow B^s_{p,q,+}(\T)
\]
is bounded;
\item for $0<p<\infty$ and $0<q\leq\infty$,
\[
 A_a:F^{s+m}_{p,q,+}(\T)\longrightarrow F^s_{p,q,+}(\T)
\]
is bounded.
\end{enumerate}
Consequently, through \eqref{eq:boundary-identification}, $A_a$ defines
bounded maps
\[
 A_a:\cB^{s+m}_{p,q}(\D)\longrightarrow\cB^s_{p,q}(\D)
\]
and
\[
 A_a:\cF^{s+m}_{p,q}(\D)\longrightarrow\cF^s_{p,q}(\D)
\]
in the same parameter ranges.
\end{theorem}

\begin{proof}
Let $f\in B^{s+m}_{p,q,+}(\T)$.  Proposition~\ref{prop:ambient} gives
\[
 \norm{A_af}_{B^s_{p,q}(\T)}
 \leq C\norm f_{B^{s+m}_{p,q}(\T)}.
\]
Proposition~\ref{prop:invariance} gives $A_af\in\cD'_+(\T)$.  Hence
$A_af\in B^s_{p,q,+}(\T)$, with the asserted bound.  The proof for
Triebel--Lizorkin spaces is identical.  Finally apply the boundary
identifications \eqref{eq:boundary-identification}.
\end{proof}

\begin{corollary}\label{cor:analytic}
If $a\in S^m_{1,0}(\T\times\Z)$ and $a(\cdot,n)$ has Fourier spectrum in
$\Zp$ for every $n\in\Zp$, then all conclusions of
Theorem~\ref{thm:main} hold.
\end{corollary}

\begin{proof}
The assumption implies $a_\nu(n)=0$ for $\nu\notin\Zp$, which is stronger than
\eqref{eq:triangular}.  Now we simply apply Theorem~\ref{thm:main}.
\end{proof}

\section{Extensions to the \texorpdfstring{$S^m_{\rho,\delta}$}{S(rho,delta)} symbol classes}\label{sec:extensions}

We now extend the preceding argument to the general toroidal
$S^m_{\rho,\delta}$ classes.  The result obtained by the maximal-function
method is uniform over the quasi-Banach range, but it is not asserted to have
the sharp loss when $\rho<1$.

\begin{definition}\label{def:rho-delta}
Let $m\in\mathbb R$ and $0\leq\delta<\rho\leq1$.  The class
$S^m_{\rho,\delta}(\T\times\Z)$ consists of all symbols, smooth in $x$, such
that
\begin{equation}\label{eq:rho-delta-symbol}
 |\Delta_n^\alpha\partial_x^\beta a(x,n)|
 \leq C_{\alpha,\beta}\la n\ra^{m-\rho|\alpha|+\delta|\beta|}
\end{equation}
for every $\alpha,\beta\in\mathbb N_0^d$.  We write
$S^m_{\rho,\delta,\triangle}$ when, in addition,
\begin{equation}\label{eq:rho-triangular}
 a_\nu(n)=0\quad\text{for }n\in\Zp\text{ and }n+\nu\notin\Zp.
\end{equation}
\end{definition}

We use the conventional notation $S^m_{\rho,\delta}(\mathbb R^d\times
\mathbb R^d)$ for the Euclidean symbols $c(x,\xi)$ satisfying
\[
 |\partial_\xi^\alpha\partial_x^\beta c(x,\xi)|
 \leq C_{\alpha,\beta}\la\xi\ra^{m-\rho|\alpha|+\delta|\beta|},
\]
and $\Op S^m_{\rho,\delta}$ for the corresponding Kohn--Nirenberg
operators.

The restriction $\delta<\rho$ is the standard range for the
Ruzhansky--Turunen symbolic calculus.  It is also the range used below to
reduce localized amplitudes to Euclidean symbols.  The exclusion of
$\delta=1$ when $\rho=1$ is substantive: general symbols of type $(1,1)$,
even of order zero, need not define bounded operators on $L^2$, as shown by
Ching's classical counterexample \cite{Ching1972}.  Additional restrictions,
such as twisted-diagonal hypotheses, are therefore needed for a comparable
$S^m_{1,1}$ theory.

\begin{proposition}\label{prop:triangular-composition}
Let $0\leq\delta<\rho\leq1$ and
$a_j\in S^{m_j}_{\rho,\delta,\triangle}(\T\times\Z)$, $j=1,2$.
Then the toroidal symbol $c$ of $A_{a_1}A_{a_2}$ belongs to
$S^{m_1+m_2}_{\rho,\delta,\triangle}(\T\times\Z)$.
\end{proposition}

\begin{proof}
The exact toroidal composition formula in the range $\delta<\rho$
from \cite{RTbook} defines $c$ by
$A_c=A_{a_1}A_{a_2}$ and gives
$c\in S^{m_1+m_2}_{\rho,\delta}(\T\times\Z)$.  Each factor preserves
$\cD'_+(\T)$ by the argument of Proposition~\ref{prop:invariance}.  Indeed,
for $a_j$, integration by parts $L$ times gives a coefficient bound of size
$\la n\ra^{m_j+\delta L}(1+|\nu|)^{-L}$; since $\delta<1$, choosing $L$
sufficiently large still makes the output-coefficient series converge.
Hence the composition preserves
$\cD'_+(\T)$.  Testing $A_c$ on $e^{i\ip n x}$, exactly as in the converse
part of Proposition~\ref{prop:invariance}, shows directly that $c$ satisfies
\eqref{eq:rho-triangular}.
\end{proof}

\begin{remark}
Proposition~\ref{prop:triangular-composition} supplies an algebra of
holomorphy-preserving operators.  A corresponding parametrix statement is
more delicate: an ambient elliptic parametrix need not preserve the positive
cone, and smoothing remainders may have infinite-dimensional defect after
restriction to the holomorphic subspace; see
Remark~\ref{rem:ellipticity}.
\end{remark}

The next estimate is recorded principally for its consequence at
$\rho=1$, where it has no derivative loss.  For $0<\rho<1$ it also gives a
direct, but generally nonsharp, maximal-function bound.

\begin{lemma}
\label{lem:rho-dyadic}
Let $a\in S^m_{\rho,\delta}(\T\times\Z)$ with
$0\leq\delta<\rho\leq1$.  Given $N>0$ and $0<r\leq1$, one has
\begin{equation}\label{eq:rho-pointwise}
 |\Delta_jA_a\Psi_kg(x)|
 \leq C_{N,r}
 2^{k[m+d(1-\rho)/r]}2^{-N|j-k|}\mathcal M_rg(x)
\end{equation}
whenever the Fourier spectrum of $g$ is contained in the dyadic support of
$\Psi_k$.
\end{lemma}

\begin{proof}
We indicate all changes to Lemma~\ref{lem:kernel}.  On the support of
$\psi_k$, the discrete Leibniz formula and
\eqref{eq:rho-delta-symbol} give
\begin{equation}\label{eq:rho-difference-product}
 \left|\Delta_n^\gamma\{a(x,n)\psi_k(n)\}\right|
 \leq C_\gamma2^{k(m-\rho|\gamma|)}.
\end{equation}
Indeed, a term containing multiorder $\mu$ differences of $a$ and
$\gamma-\mu$ differences of $\psi_k$ is at most
\[
 C2^{k(m-\rho|\mu|)}2^{-k(|\gamma|-|\mu|)}
 \leq C2^{k(m-\rho|\gamma|)}.
\]
Repeated summation by parts in \eqref{eq:basic-kernel} therefore yields
\begin{equation}\label{eq:rho-basic-kernel}
 |H_k(x,y)|
 \leq C_L2^{k(m+d)}
 (1+2^{\rho k}\tdist(x,y))^{-L}.
\end{equation}

The output projection gives the same off-diagonal factor as before.  If
$|j-k|\geq4$, the output constraint $|n+\nu|\asymp2^j$ and input constraint
$|n|\asymp2^k$ imply
\[
 |\nu|\geq c2^{\max(j,k)}.
\]
Integration by parts $M$ times in $x$ gives
\begin{equation}\label{eq:rho-fourier-coeff}
 |a_\nu(n)|
 \leq C_M(1+|\nu|)^{-M}\la n\ra^{m+\delta M}
 \leq C_M2^{km}2^{-M\max(j,k)+\delta Mk}.
\end{equation}
When $j\geq k+4$, this contains $2^{-M(j-k)}2^{-(1-\delta)Mk}$; when
$k\geq j+4$, it contains $2^{-(1-\delta)Mk}$.  Because $\delta<1$, $M$ can
be chosen so large that either expression is bounded by
$C_N2^{-N|j-k|}$, with enough residual decay to carry out the spatial
summation-by-parts argument based on \eqref{eq:rho-difference-product}.
Combining the diagonal and off-diagonal cases gives the kernel
bound
\begin{equation}\label{eq:rho-full-kernel}
 |K_{j,k}(x,y)|
 \leq C_{N,L}2^{k(m+d)}2^{-N|j-k|}
 (1+2^{\rho k}\tdist(x,y))^{-L}.
\end{equation}

It remains to convert \eqref{eq:rho-full-kernel} to a maximal estimate.  Put
\[
 E_h=\{y:2^{h-1-\rho k}\leq\tdist(x,y)<2^{h-\rho k}\},
 \qquad h\geq0,
\]
with the evident interpretation for $h=0$.  The band-limited submean
estimate \eqref{eq:local-integral}, applied at bandwidth $2^k$ with
\[
 u=h+(1-\rho)k,
\]
implies
\begin{equation}\label{eq:rho-local-integral}
 \int_{E_h}|g(y)|\dd y
 \leq C2^{hd/r-kd+d(1-\rho)k/r}\mathcal M_rg(x).
\end{equation}
Indeed, the radius $2^{h-\rho k}$ equals $2^{u-k}$ and
\[
 2^{ud/r-kd}=2^{hd/r-kd+d(1-\rho)k/r}.
\]
Thus the factor $2^{d(1-\rho)k/r}$ is exactly the cost of controlling a
function of bandwidth $2^k$ on a kernel ball of the larger radius
$2^{-\rho k}$.
On $E_h$, the right side of \eqref{eq:rho-full-kernel} is at most
\[
 C2^{k(m+d)}2^{-N|j-k|}2^{-hL}.
\]
Multiplication by \eqref{eq:rho-local-integral}, followed by summation in
$h$, gives
\[
 C2^{k[m+d(1-\rho)/r]}2^{-N|j-k|}\mathcal M_rg(x)
 \sum_{h\geq0}2^{-h(L-d/r)}.
\]
Choose $L>d/r$.  This proves \eqref{eq:rho-pointwise}.
\end{proof}

The dyadic estimate will be used below for the endpoint $\rho=1$.  For
$0<\rho<1$, the main non-endpoint result instead follows by localization of
the sharp Euclidean theorem.  Set
\begin{equation}\label{eq:critical-loss}
 \sigma_p(\rho):=d(1-\rho)\left|\frac1p-\frac12\right|.
\end{equation}
Here and below $1/\infty$ is understood to be zero.

\begin{lemma}\label{lem:rho-localization}
Let $0<\rho<1$, $0\leq\delta<\rho$, and
$a\in S^m_{\rho,\delta}(\T\times\Z)$.  There are a properly supported
operator $A_0$ and a smoothing operator $R$ on $\T$ such that
$A_a=A_0+R$.  Moreover, there are finitely many cutoffs
$\eta_\mu,\theta_\mu\in C^\infty(\T)$ satisfying
\[
 \sum_\mu\eta_\mu=1,
 \qquad
 \eta_\mu A_0f=\eta_\mu A_0(\theta_\mu f),
\]
for which every localized operator $\eta_\mu A_0\theta_\mu$, lifted to a
coordinate chart in $\mathbb R^d$, belongs to
$\Op S^m_{\rho,\delta}(\mathbb R^d\times\mathbb R^d)$, with each of its
symbol seminorms controlled by finitely many seminorms of $a$.
\end{lemma}

\begin{proof}
The distribution kernel of $A_a$ is
\[
 K_a(x,y)=\sum_{n\in\Z}e^{i\ip n{(x-y)}}a(x,n).
\]
It is smooth away from the diagonal.  Indeed, on a compact subset of
$(\T\times\T)\setminus\operatorname{diag}$, repeated summation by parts in
$n$ expresses every $x,y$ derivative of the kernel in terms of a series
whose summands are bounded by
\[
 C_{\beta,\gamma,N}\la n\ra^{m+|\beta|+|\gamma|-\rho N}.
\]
Here $x,y$ derivatives contribute at most $|\beta|+|\gamma|$ powers of
$\la n\ra$, whereas the $N$ summation-by-parts differences supply the
displayed decay; differences falling on polynomial phase factors only
improve this bound.  Choosing $N$ sufficiently large
proves convergence with all derivatives.

Choose $\zeta\in C^\infty(\T\times\T)$ equal to one near the diagonal and
supported where $\tdist(x,y)<2\varepsilon<\pi$.  Let $A_0$ have kernel
$\zeta K_a$ and let $R$ have kernel $(1-\zeta)K_a$.  The preceding paragraph
shows that $R$ is smoothing, i.e.,  the operator has a $C^\infty$ kernel and 
it maps $\cD'(\T)$ continuously into $C^\infty(\T)$.

Take a finite partition of unity $\sum_\mu\eta_\mu=1$, with each
$\eta_\mu$ compactly supported in a coordinate chart, and choose
$\theta_\mu$ equal to one on the
$2\varepsilon$-neighbourhood of $\supp\eta_\mu$.  Proper support gives
$\eta_\mu A_0f=\eta_\mu A_0(\theta_\mu f)$.  By the extension theorem for
toroidal symbols \cite[Theorem~5.2]{RT2010}, $a$ is the restriction to
$\Z$ of a periodic Euclidean symbol
\[
 \widetilde a\in S^m_{\rho,\delta}(\T\times\mathbb R^d).
\]
The periodization identity \cite[Theorem~6.4]{RT2010} identifies the
toroidal kernel locally with the Euclidean kernel of $\widetilde a$.  The
terms coming from nonzero translates of the coordinate chart are smooth,
because they stay away from the diagonal.  Thus, modulo a smoothing
operator, the localized operator has the compactly supported amplitude
\[
 c_\mu(x,y,\xi)
 =\eta_\mu(x)\zeta(x,y)\widetilde a(x,\xi)\theta_\mu(y).
\]
It satisfies
\[
 |\partial_\xi^\alpha\partial_x^\beta\partial_y^\gamma
   c_\mu(x,y,\xi)|
 \leq C_{\alpha,\beta,\gamma}
 \la\xi\ra^{m-\rho|\alpha|+\delta(|\beta|+|\gamma|)}.
\]
Since $\delta<\rho$, the standard amplitude-reduction theorem places the
corresponding operator in
$\Op S^m_{\rho,\delta}(\mathbb R^d\times\mathbb R^d)$, modulo a smoothing
operator, which may be absorbed into its symbol; see also the compatible
compound-symbol reduction in \cite[Appendix~A]{Park2018}.   
 We notice that the extension, cutoff, and amplitude-reduction constructions are
continuous in their respective symbol topologies, so every resulting
seminorm is controlled by finitely many toroidal seminorms of $a$. This completes the proof of the
lemma.
\end{proof}

With this local reduction to the Euclidean calculus in hand, the sharp
Euclidean $(\rho,\delta)$ estimates of Park \cite{Park2018} transfer
directly to the strict range $\kappa>\sigma_p(\rho)$.

\begin{proposition}\label{prop:rho-ambient}
Let $0<\rho<1$, $0\leq\delta<\rho$, and
$a\in S^m_{\rho,\delta}(\T\times\Z)$.  If
\begin{equation}\label{eq:strict-critical-loss}
 \kappa>\sigma_p(\rho),
\end{equation}
then, for every $s\in\mathbb R$,
\begin{align}
 A_a&:B^{s+m+\kappa}_{p,q}(\T)\longrightarrow B^s_{p,q}(\T),
 &&0<p,q\leq\infty,\label{eq:sharp-periodic-B}\\
 A_a&:F^{s+m+\kappa}_{p,q}(\T)\longrightarrow F^s_{p,q}(\T),
 &&0<p<\infty,\quad0<q\leq\infty.\label{eq:sharp-periodic-F}
\end{align}
\end{proposition}

\begin{proof}
We apply Lemma~\ref{lem:rho-localization}.  A smooth kernel maps every periodic
Besov or Triebel--Lizorkin space continuously into every smoother space, so
the remainder $R$ is harmless.

For a cutoff supported compactly inside one coordinate chart, periodic
and Euclidean norms are locally equivalent. More precisely, choose the
partition of unity so that
\(\operatorname{supp}\eta_\mu\Subset U_\mu\), where
\(\kappa_\mu:U_\mu\to V_\mu\subset\mathbb R^d\) is a coordinate chart.
We write
\((\eta_\mu h)^{\kappa_\mu}:=(\eta_\mu h)\circ\kappa_\mu^{-1}\) on
\(V_\mu\), extended by zero to \(\mathbb R^d\). Then,
for \(Y\in\{B,F\}\) and \(u\in\mathbb R\),
\begin{equation}\label{eq:periodic-local-norms}
 \|h\|_{Y^u_{p,q}(\mathbb T^d)}
 \asymp
 \left(
   \sum_\mu
   \|(\eta_\mu h)^{\kappa_\mu}\|_
      {Y^u_{p,q}(\mathbb R^d)}^\tau
 \right)^{1/\tau},
 \qquad
 0<\tau\leq\min(1,p,q),
\end{equation}
where \(\min(1,p,\infty)=\min(1,p)\). The constants depend only on the
fixed finite atlas, the partition of unity, and the chosen
Littlewood--Paley resolutions.

Indeed, using the auxiliary projections \(\Psi_k\), the rapid decay of
the Littlewood--Paley kernels, and the Peetre maximal estimate from
Lemma~\ref{lem:max-domination}, one obtains the localized almost-diagonal estimates needed to
compare the periodic and Euclidean Littlewood--Paley pieces. Discrete
convolution in the scale index, followed by the scalar or vector-valued
maximal inequality exactly as in Proposition~\ref{prop:ambient}, gives
\[
 \|(\eta_\mu h)^{\kappa_\mu}\|_
    {Y^u_{p,q}(\mathbb R^d)}
 \leq
 C_\mu\|h\|_{Y^u_{p,q}(\mathbb T^d)}.
\]
The converse follows by the analogous argument with the periodic and
Euclidean Littlewood--Paley projections interchanged, using a cutoff
equal to one on a neighbourhood of
\(\operatorname{supp}\eta_\mu\). Finally, since the atlas is finite, all
positive sequence quasi-norms on the chart index set are equivalent.
This proves \eqref{eq:periodic-local-norms}.

Let $T_\mu$ be the Euclidean lift of $\eta_\mu A_0\theta_\mu$.  The preceding
lemma gives
\[
 T_\mu\in\Op S^m_{\rho,\delta}(\mathbb R^d\times\mathbb R^d)
 \subseteq\Op S^m_{\rho,\rho}(\mathbb R^d\times\mathbb R^d).
\]
Apply the strict part of the sharp Euclidean theorem
\cite[Theorems~1.1--1.4]{Park2018} with
\[
 s_1=s+m+\kappa,\qquad s_2=s.
\]
Its order condition holds because
\[
 m-s_1+s_2=-\kappa
 <-d(1-\rho)\left|\frac1p-\frac12\right|.
\]
Taking equal source and target fine indices gives
\[
 \|T_\mu g\|_{Y^s_{p,q}(\mathbb R^d)}
 \leq C_\mu\|g\|_{Y^{s+m+\kappa}_{p,q}(\mathbb R^d)},
 \qquad Y\in\{B,F\}.
\]
Consequently,
\[
 \|\eta_\mu A_0f\|_{Y^s_{p,q}(\mathbb R^d)}
 \leq C_\mu
 \|\theta_\mu f\|_{Y^{s+m+\kappa}_{p,q}(\mathbb R^d)}.
\]
We now sum the finitely many estimates using \eqref{eq:periodic-local-norms} and
the boundedness of multiplication by $\theta_\mu$.  This proves
\eqref{eq:sharp-periodic-B} and \eqref{eq:sharp-periodic-F} for $A_0$.
Adding the smoothing remainder completes the proof.
\end{proof}

The strict estimate just proved excludes the critical loss
$\kappa=\sigma_p(\rho)$.  We now record the endpoint case, in which
boundedness persists at $\kappa=\sigma_p(\rho)$ itself, at the cost of a
possible change of the secondary index.

\begin{proposition}
\label{prop:rho-endpoint}
Let $0<\rho<1$, $0\leq\delta<\rho$, and
$a\in S^m_{\rho,\delta}(\T\times\Z)$.  For every $s\in\mathbb R$ the
following endpoint estimates hold.
\begin{enumerate}[label=\textup{(\roman*)}]
\item If $0<p\leq\infty$ and $0<q\leq t\leq\infty$, then
\begin{equation}\label{eq:endpoint-periodic-B}
 A_a:B^{s+m+\sigma_p(\rho)}_{p,q}(\T)
 \longrightarrow B^s_{p,t}(\T).
\end{equation}
\item If $0<p<\infty$ and $0<q,t\leq\infty$, then
\begin{equation}\label{eq:endpoint-periodic-F}
 A_a:F^{s+m+\sigma_p(\rho)}_{p,q}(\T)
 \longrightarrow F^s_{p,t}(\T)
\end{equation}
provided one of the following conditions holds:
\begin{align*}
 &0<p<2, &&p\leq t,\quad q\text{ arbitrary};\\
 &p=2,   &&q\leq2\leq t;\\
 &2<p<\infty, &&q\leq p,\quad t\text{ arbitrary}.
\end{align*}
\end{enumerate}
\end{proposition}

\begin{proof}
Use the decomposition and cutoffs in
Lemma~\ref{lem:rho-localization}.  For $p\ne2$, apply the endpoint parts of
\cite[Theorems~1.1--1.4]{Park2018} to each Euclidean lift $T_\mu$, with
\[
 s_1=s+m+\sigma_p(\rho),\qquad s_2=s.
\]
The critical order identity is
\[
 m-s_1+s_2=-\sigma_p(\rho)
 =-d(1-\rho)\left|\frac1p-\frac12\right|.
\]
Park's Besov endpoint requires $q\leq t$.  Its Triebel--Lizorkin endpoint
requires $p\leq t$ when $p<2$ and $q\leq p$ when $p>2$, exactly as stated.

At $p=2$ one has $\sigma_2(\rho)=0$, and Park states the central endpoint in
an order-zero normalization.  Let $\Lambda=(1-\Delta_x)^{1/2}$ and set
\[
 U_\mu=\Lambda^sT_\mu\Lambda^{-(s+m)}.
\]
The composition theorem for $\delta<\rho$ gives, modulo a smoothing
operator,
$U_\mu\in\Op S^0_{\rho,\delta}\subseteq\Op S^0_{\rho,\rho}$.  Park's
order-zero result yields
\[
 U_\mu:F^0_{2,q}(\mathbb R^d)\longrightarrow F^0_{2,t}(\mathbb R^d),
 \qquad q\leq2\leq t,
\]
which is equivalent to the asserted estimate for $T_\mu$.

In every case the local estimate has the form
\[
 \|\eta_\mu A_0f\|_{Y^s_{p,t}(\mathbb R^d)}
 \leq C_\mu
 \|\theta_\mu f\|_{Y^{s+m+\sigma_p(\rho)}_{p,q}(\mathbb R^d)},
 \qquad Y\in\{B,F\}.
\]
We apply the local norm equivalence \eqref{eq:periodic-local-norms} separately
with the source index $q$ and target index $t$, and sum over the finite
cover.  The smoothing remainder is bounded between all the spaces in
question.  This proves both assertions.
\end{proof}

For equal source and target secondary indices, Proposition~\ref{prop:rho-endpoint}
gives the endpoint map on every $B^s_{p,q}$.  On $F^s_{p,q}$ it gives the
endpoint map precisely in the ranges
\begin{equation}\label{eq:endpoint-equal-indices}
 q\geq p\quad(0<p<2),\qquad
 q=2\quad(p=2),\qquad
 q\leq p\quad(2<p<\infty).
\end{equation}

\begin{theorem}
\label{thm:rho-holomorphic}
Let $a\in S^m_{\rho,\delta,\triangle}(\T\times\Z)$.
For $0\leq\delta<\rho<1$, the conclusions of Proposition~\ref{prop:rho-ambient} hold with every
periodic space replaced by its positive-frequency subspace.  Equivalently,
for the same ranges of parameters,
\begin{align*}
 A_a&:\cB^{s+m+\kappa}_{p,q}(\D)\longrightarrow
       \cB^s_{p,q}(\D),
 &&\kappa>\sigma_p(\rho),\\
 A_a&:\cF^{s+m+\kappa}_{p,q}(\D)\longrightarrow
       \cF^s_{p,q}(\D),
 &&\kappa>\sigma_p(\rho).
\end{align*}
For $\rho=1$ and $0\leq\delta<1$, no additional loss is necessary:
Theorem~\ref{thm:main} holds with $S^m_{1,0,\triangle}$ replaced by
$S^m_{1,\delta,\triangle}$.
\end{theorem}

\begin{proof}
For $\rho<1$, Proposition~\ref{prop:rho-ambient} gives the ambient norm
estimate.  Proposition~\ref{prop:invariance} remains valid verbatim because
its proof uses only the triangular condition and the rapid decay of the
$x$-Fourier coefficients for each fixed symbol order.  Restricting the
ambient operator to $\cD'_+$ and using
\eqref{eq:boundary-identification} proves the two displayed maps.

If $\rho=1$, the factor $2^{kd(1-\rho)/r}$ in
\eqref{eq:rho-pointwise} is one.  The proof of
Proposition~\ref{prop:ambient} therefore applies with $\kappa=0$.
The assumption $\delta<1$ supplies the off-diagonal decay in
\eqref{eq:rho-fourier-coeff}.  Positive-frequency invariance and the boundary
identification finish the proof.
\end{proof}
We have the following corollary.
\begin{corollary}
\label{cor:rho-holomorphic-endpoint}
Let $0<\rho<1$, $0\leq\delta<\rho$, and
$a\in S^m_{\rho,\delta,\triangle}(\T\times\Z)$.  Every conclusion of
Proposition~\ref{prop:rho-endpoint} holds on the corresponding
positive-frequency subspaces.  Equivalently,
\begin{align*}
 A_a&:\cB^{s+m+\sigma_p(\rho)}_{p,q}(\D)
       \longrightarrow\cB^s_{p,t}(\D),
 &&0<p\leq\infty,\quad0<q\leq t\leq\infty,\\
 A_a&:\cF^{s+m+\sigma_p(\rho)}_{p,q}(\D)
       \longrightarrow\cF^s_{p,t}(\D),
\end{align*}
where the second map holds under the three Triebel--Lizorkin alternatives in
Proposition~\ref{prop:rho-endpoint}.
\end{corollary}

\begin{proof}
Proposition~\ref{prop:rho-endpoint} supplies the ambient bounds, while
Proposition~\ref{prop:invariance} supplies positive-frequency invariance.
Restriction to the positive-frequency subspaces and
\eqref{eq:boundary-identification} give the two holomorphic estimates.
\end{proof}

We finish the subsection by showing that the loss in
Proposition~\ref{prop:rho-ambient} is sharp even after restriction to the
positive cone.  The same example also gives necessity of the Besov
fine-index condition at the endpoint.

\begin{lemma}\label{lem:oscillatory-lower}
Let $0<\rho<1$, $m\in\mathbb R$, and
\[
 b(n)=\la n\ra^m\exp\bigl(i\la n\ra^{1-\rho}\bigr),
 \qquad n\in\Z.
\]
Then $b\in S^m_{\rho,0,\triangle}(\T\times\Z)$.  For every
$0<p\leq\infty$ and every sufficiently large dyadic $R$ there is a
trigonometric polynomial $g_R$ with Fourier spectrum in
\[
 \Zp\cap\{n:cR\leq|n|\leq CR\}
\]
such that
\begin{equation}\label{eq:oscillatory-lower}
 \|b(D)g_R\|_{L^p(\T)}
 \geq c_p R^{m+\sigma_p(\rho)}\|g_R\|_{L^p(\T)}.
\end{equation}
The constants $c,C$ and $c_p>0$ are independent of $R$.
\end{lemma}

\begin{proof}
Finite differences of the restriction of
$\la\xi\ra^m e^{i\la\xi\ra^{1-\rho}}$ satisfy
\[
 |\Delta_n^\alpha b(n)|\leq C_\alpha
 \la n\ra^{m-\rho|\alpha|},
\]
which follows either from the mean-value formula for finite differences or
by differentiating the Euclidean function.  Thus $b\in S^m_{\rho,0}$, and
it is triangular because it is independent of $x$.

We prove the lower estimate first for $0<p<2$.  Choose a point
$\xi_0\in(0,\infty)^d$ and a nonnegative
$\chi\in C_c^\infty(\mathbb R^d)$ with $\chi(\xi_0)>0$ and whose
support is contained in a sufficiently small annular neighbourhood of
$\xi_0$ inside a cone $\Gamma\Subset(0,\infty)^d$.  Set
\[
 g_R(x)=\sum_{n\in\Z}\chi(n/R)e^{i\ip n x}.
\]
Only positive-cone frequencies occur when $R$ is large.  Poisson summation
gives, up to the fixed Fourier-normalization constant,
\[
 g_R(x)=R^d\sum_{\ell\in\Z}\widehat\chi\bigl(R(x+2\pi\ell)\bigr).
\]
Rapid decay gives the upper bound below for $p\geq1$ and for $p=\infty$.
For $0<p<1$, use $(\sum_\ell c_\ell)^p\leq\sum_\ell c_\ell^p$ and integrate
over a fundamental domain; the translates then tile $\mathbb R^d$.  For the
lower bound, $\widehat\chi(0)>0$, so on a ball of radius $c/R$ the term
$\ell=0$ is bounded below by $cR^d$, while all other terms are negligible
by rapid decay.  Consequently
\begin{equation}\label{eq:input-packet-size}
 \|g_R\|_{L^p(\T)}\asymp R^{d(1-1/p)},
 \qquad 0<p\leq\infty.
\end{equation}
For the output, write $x=R^{-\rho}y$.  After extracting the factor
$R^{m+d}$, the zero term in Poisson summation has amplitude
\[
 A_R(\xi)=R^{-m}\la R\xi\ra^m\chi(\xi)
\]
and large parameter $R^{1-\rho}$ multiplying the phase
\[
 \Phi_{R,y}(\xi)
 =R^{\rho-1}\la R\xi\ra^{1-\rho}+y\mathbin\cdot\xi.
\]
Uniformly for $y$ in compact sets, $A_R$ converges in
$C^\infty(\supp\chi)$ to $|\xi|^m\chi(\xi)$ and $\Phi_{R,y}$ converges to
\[
 \Phi_y(\xi)=|\xi|^{1-\rho}+y\mathbin\cdot\xi.
\]
After shrinking the support of $\chi$, the map
\[
 \xi\longmapsto-\nabla|\xi|^{1-\rho}
\]
is a diffeomorphism on a neighbourhood of $\supp\chi$.  Indeed,
\[
 \nabla|\xi|^{1-\rho}=(1-\rho)|\xi|^{-1-\rho}\xi
\]
and
\[
 D^2|\xi|^{1-\rho}
 =(1-\rho)|\xi|^{-1-\rho}\operatorname{Id}_{d\times d}
 -(1-\rho^2)|\xi|^{-3-\rho}\xi\otimes\xi.
\]
The tangential eigenvalues are $(1-\rho)|\xi|^{-1-\rho}$ and the radial
eigenvalue is $-\rho(1-\rho)|\xi|^{-1-\rho}$.  Thus
\[
 \det D^2|\xi|^{1-\rho}
 =-\rho(1-\rho)^d|\xi|^{-d(1+\rho)}\ne0,
 \qquad \xi\ne0.
\]
The phase is therefore uniformly nondegenerate on $\supp\chi$.  Taking a
relatively compact neighbourhood $U$ of
$-\nabla|\xi_0|^{1-\rho}$ sufficiently small ensures that the unique
critical point lies in a fixed compact subset on which $\chi$ is bounded
below.  The leading stationary-phase coefficient is
therefore bounded away from zero uniformly for $y\in U$, while the uniform
remainder is of strictly lower order.  Uniform stationary phase then gives,
for all sufficiently large $R$,
\begin{equation}\label{eq:stationary-packet-size}
 |b(D)g_R(R^{-\rho}y)|
 \geq cR^{m+d-d(1-\rho)/2},\qquad y\in U.
\end{equation}
For a nonzero Poisson index $\ell$, the unscaled phase contains
$2\pi R\ell\mathbin\cdot\xi$; its gradient is bounded below by
$cR|\ell|$ uniformly for $y\in U$.  Integration by parts and summation in
$\ell$ show that the total contribution of the nonzero terms is
$O(R^{-N})$ for every $N$.  Since the set $R^{-\rho}U$ has measure comparable to
$R^{-\rho d}$, \eqref{eq:stationary-packet-size} yields
\[
 \|b(D)g_R\|_{L^p}
 \gtrsim R^{m+d-d(1-\rho)/2-\rho d/p}.
\]
Dividing by \eqref{eq:input-packet-size} gives
\eqref{eq:oscillatory-lower}, because for $p<2$
\[
 d-d(1-\rho)/2-\rho d/p-d(1-1/p)
 =d(1-\rho)(1/p-1/2).
\]
We now treat $2<p\leq\infty$ by duality.  Put $p'=p/(p-1)$, with
$p'=1$ when $p=\infty$, and let
\[
 b^*(n)=\overline{b(n)}
 =\la n\ra^m\exp\bigl(-i\la n\ra^{1-\rho}\bigr).
\]
The preceding construction is unchanged when the sign of the phase is
reversed.  Since $1\leq p'<2$, it supplies a polynomial $h_R$, supported in
$R\Gamma\cap\Zp$, such that
\begin{equation}\label{eq:dual-packet-lower}
 \|b^*(D)h_R\|_{L^{p'}}
 \geq cR^{m+\sigma_{p'}(\rho)}\|h_R\|_{L^{p'}}.
\end{equation}
Choose a real function $\vartheta\in C_c^\infty(\Gamma)$ which equals one
on a neighbourhood of $\supp\chi$, and let
\[
 Q_Rf=\sum_{n\in\Z}\vartheta(n/R)\wh f(n)e^{i\ip n x}.
\]
Its convolution kernel has uniformly bounded $L^1$ norm, by Poisson
summation, so
\begin{equation}\label{eq:uniform-annular-projection}
 \|Q_Rf\|_{L^p}\leq C_p\|f\|_{L^p},
 \qquad 1\leq p\leq\infty,
\end{equation}
with a constant independent of $R$.  Moreover, $Q_R$ is self-adjoint,
$Q_Rb^*(D)h_R=b^*(D)h_R$, and its range has Fourier spectrum in $\Zp$ for
large $R$.

By $L^p$--$L^{p'}$ duality, choose $u_R\in L^p(\T)$ with
$\|u_R\|_{L^p}=1$ such that
\[
 \left|\int_\T u_R(x)\overline{b^*(D)h_R(x)}
       \frac{\dd x}{(2\pi)^d}\right|
 \geq\frac12\|b^*(D)h_R\|_{L^{p'}}.
\]
For $p=\infty$, one may take the pointwise complex sign of
$b^*(D)h_R$, so the factor $1/2$ is unnecessary.  Set
$f_R=Q_Ru_R$.  Then \eqref{eq:uniform-annular-projection} gives
$\|f_R\|_{L^p}\leq C_p$, while self-adjointness of $Q_R$ and the adjoint
identity $b(D)^*=b^*(D)$ give
\begin{align*}
 \|b(D)f_R\|_{L^p}\|h_R\|_{L^{p'}}
 &\geq
 \left|\int_\T b(D)f_R(x)\overline{h_R(x)}
       \frac{\dd x}{(2\pi)^d}\right|\\
 &=\left|\int_\T u_R(x)\overline{b^*(D)h_R(x)}
       \frac{\dd x}{(2\pi)^d}\right|\\
 &\geq\frac12\|b^*(D)h_R\|_{L^{p'}}.
\end{align*}
After dividing by $\|h_R\|_{L^{p'}}$ and then by the uniformly bounded
$\|f_R\|_{L^p}$, \eqref{eq:dual-packet-lower} proves
\eqref{eq:oscillatory-lower}, because
$\sigma_{p'}(\rho)=\sigma_p(\rho)$.  The polynomial $f_R$ has the required
positive-cone annular spectrum.
At $p=2$, a single exponential with
frequency in the relevant set gives the assertion, since $\sigma_2(\rho)=0$.
\end{proof}
We have the following sharpness result.
\begin{theorem}\label{thm:periodic-sharpness}
Let $0<\rho<1$, $0\leq\delta<\rho$, $m,s\in\mathbb R$, and
$0<p\leq\infty$.
\begin{enumerate}[label=\textup{(\roman*)}]
\item If, for every $a\in S^m_{\rho,\delta,\triangle}$, one has either
\[
 A_a:B^{s+m+\kappa}_{p,q,+}(\T)\longrightarrow B^s_{p,t,+}(\T)
\]
or, when $p<\infty$,
\[
 A_a:F^{s+m+\kappa}_{p,q,+}(\T)\longrightarrow F^s_{p,t,+}(\T),
\]
then $\kappa\geq\sigma_p(\rho)$.
\item At $\kappa=\sigma_p(\rho)$, boundedness
\[
 A_a:B^{s+m+\sigma_p(\rho)}_{p,q,+}(\T)
 \longrightarrow B^s_{p,t,+}(\T)
\]
for every $a\in S^m_{\rho,\delta,\triangle}$ is possible only if
$q\leq t$.
\end{enumerate}
Both conclusions hold, through boundary values, on the corresponding
holomorphic spaces on $\D$.
\end{theorem}

\begin{proof}
Use the multiplier $b$ of Lemma~\ref{lem:oscillatory-lower}, which belongs
to $S^m_{\rho,0,\triangle}\subset S^m_{\rho,\delta,\triangle}$.  On a
single dyadic annulus the Besov and Triebel--Lizorkin quasi-norms both reduce
to $R^u$ times the $L^p$ quasi-norm.  Consequently
\eqref{eq:oscillatory-lower} contradicts either map in part~(i) when
$\kappa<\sigma_p(\rho)$.

For part~(ii), choose the dyadic sequence $R_j$ sufficiently lacunary that
the Fourier supports of the corresponding $g_{R_j}$ lie in disjoint
Littlewood--Paley annuli and the associated projection equals one on each
support.  The outputs $b(D)g_{R_j}$ have the same separation because
$b(D)$ is a multiplier.  Normalize $\|g_{R_j}\|_{L^p}=1$.  For every finitely supported scalar
sequence $\{\lambda_j\}$, put
\[
 f=\sum_j\lambda_jR_j^{-s-m-\sigma_p(\rho)}g_{R_j}.
\]
The Littlewood--Paley definition now shows that the source Besov quasi-norm is comparable to
$\|\lambda\|_{\ell^q}$, whereas
Lemma~\ref{lem:oscillatory-lower} gives
\[
 \|b(D)f\|_{B^s_{p,t}}\gtrsim\|\lambda\|_{\ell^t}.
\]
Thus boundedness would imply the embedding $\ell^q\hookrightarrow\ell^t$, which is
equivalent to $q\leq t$.  All test functions have spectrum in $\Zp$, and
the boundary identifications \eqref{eq:boundary-identification} transfer
the failures to the holomorphic spaces.
\end{proof}

\begin{remark}
Theorem~\ref{thm:periodic-sharpness} proves sharpness of the derivative loss
and the Besov endpoint restriction on both the periodic positive-cone and
holomorphic scales.  It does not establish the necessity of all the
Triebel--Lizorkin secondary-index alternatives in
Proposition~\ref{prop:rho-endpoint}; adapting the spatial overlap and
randomization portions of Park's Euclidean examples to the positive lattice
would require a separate argument.  We leave this question open.
\end{remark}

\section{Application to a holomorphic differential equation}\label{sec:application}
We conclude this paper by applying the $S^m_{1,0}$ part of the main
boundedness result to a first-order operator whose associated symbol can be
written down explicitly.

Fix $\lambda\in\mathbb C$ and consider
\begin{equation}\label{eq:application-L}
 L_\lambda u(z):=J_du(z)+\lambda z_1u(z),
 \qquad
 J_d:=d+\sum_{j=1}^dz_j\partial_{z_j}.
\end{equation}
Writing $|n|_+=n_1+\cdots+n_d$ for $n\in\Zp$, one has
\[
 J_du=\sum_{n\in\Zp}(d+|n|_+)\wh u(n)z^n.
\]
Thus $L_\lambda=J_d+\lambda M_{z_1}$.  On the distinguished boundary
$z_j=e^{ix_j}$, this is the toroidal operator with symbol
\begin{equation}\label{eq:application-symbol}
 a_\lambda(x,n)=d+n_1+\cdots+n_d+\lambda e^{ix_1},
 \qquad n\in\Zp.
\end{equation}
The same formula defines its extension to every $n\in\Z$; only its
restriction to $\Zp$ acts on the holomorphic subspace.
The symbol belongs to $S^1_{1,0,\triangle}$.  Indeed,
\[
 |\Delta_n^\alpha a_\lambda(x,n)|\leq C_\alpha
 \la n\ra^{1-|\alpha|},
\]
and every positive-order $x$-derivative is a constant multiple of
$e^{ix_1}$.  Moreover, its only $x$-Fourier frequencies are $0$ and $e_1$,
so it is
analytic in $x$ and hence triangular.  Theorem~\ref{thm:main} implies
\begin{align}
 L_\lambda&:\cB^{s+1}_{p,q}(\D)\longrightarrow\cB^s_{p,q}(\D),
 &&0<p,q\leq\infty,
 \label{eq:application-B-forward}\\
 L_\lambda&:\cF^{s+1}_{p,q}(\D)\longrightarrow\cF^s_{p,q}(\D),
 &&0<p<\infty,\quad0<q\leq\infty.
 \label{eq:application-F-forward}
\end{align}

For the remainder of this section, $\cX^s_{p,q}(\D)$ denotes either
$\cB^s_{p,q}(\D)$, in the range $0<p,q\leq\infty$, or
$\cF^s_{p,q}(\D)$, in the range $0<p<\infty$ and $0<q\leq\infty$.

\begin{proposition}
\label{prop:application}
For every $\lambda\in\mathbb C$ and every admissible $s,p,q$, the operator
$L_\lambda$ is a topological isomorphism
\[
 L_\lambda:\cX^{s+1}_{p,q}(\D)\longrightarrow\cX^s_{p,q}(\D).
\]
Hence the equation
\begin{equation}\label{eq:application-equation}
 L_\lambda u=f
\end{equation}
has a unique solution $u\in\cX^{s+1}_{p,q}(\D)$ for every
$f\in\cX^s_{p,q}(\D)$, and
\begin{equation}\label{eq:application-estimate}
 \norm{u}_{\cX^{s+1}_{p,q}}
 \leq C_{\lambda,s,p,q}\norm{f}_{\cX^s_{p,q}}.
\end{equation}
\end{proposition}

\begin{proof}
The identity
\[
 M_{e^{-\lambda z_1}}J_dM_{e^{\lambda z_1}}
 =J_d+\lambda M_{z_1}=L_\lambda
\]
gives
\begin{equation}\label{eq:application-solution}
 u(z)=e^{-\lambda z_1}\int_0^1
 t^{d-1}e^{\lambda t z_1}f(tz)\dd t.
\end{equation}

Define
\[
 J_d^{-1}h(z):=\int_0^1t^{d-1}h(tz)\dd t
 =\sum_{n\in\Zp}\frac{\wh h(n)}{d+|n|_+}z^n.
\]
Formula \eqref{eq:application-solution} gives the factorization
\begin{equation}\label{eq:application-inverse}
 L_\lambda^{-1}
 =M_{e^{-\lambda z_1}}J_d^{-1}M_{e^{\lambda z_1}}.
\end{equation}
The multiplication operators $M_{e^{\pm\lambda z_1}}$ have toroidal symbols
\[
 b_\pm(x,n)=e^{\pm\lambda e^{ix_1}}.
\]
These symbols are independent of $n$, smooth and analytic in $x$, and hence
belong to $S^0_{1,0,\triangle}$.  Corollary~\ref{cor:analytic} therefore
shows that the two multiplication operators are bounded on every
$\cX^s_{p,q}(\D)$.

For $n\in\Z$, put $\|n\|_1:=|n_1|+\cdots+|n_d|$.  The multiplier
$b(n)=(d+\|n\|_1)^{-1}$ belongs to $S^{-1}_{1,0,\triangle}$ and agrees with
$(d+n_1+\cdots+n_d)^{-1}$ on
positive-frequency distributions.  It is triangular because it is
independent of $x$.  Theorem~\ref{thm:main} for an order
$m=-1$ symbol therefore states
\[
 J_d^{-1}:\cX^{t-1}_{p,q}(\D)\longrightarrow\cX^t_{p,q}(\D).
\]
Taking $t=s+1$ gives precisely
\[
 J_d^{-1}:\cX^s_{p,q}(\D)\longrightarrow
        \cX^{s+1}_{p,q}(\D).
\]
The factorization proves \eqref{eq:application-estimate} and existence.

For uniqueness, suppose $L_\lambda u=0$.  Then
$J_d(e^{\lambda z_1}u)=0$.  Since $J_d$ multiplies the Taylor coefficient
of $z^n$ by the strictly positive number $d+|n|_+$, one has $u=0$.  Finally,
\eqref{eq:application-B-forward}--\eqref{eq:application-F-forward} give
continuity of $L_\lambda$, while \eqref{eq:application-inverse} gives
continuity of its inverse.  Thus it is a topological isomorphism.
\end{proof}

\begin{remark}\label{rem:ellipticity}
The preceding isomorphism does not follow from toroidal ellipticity alone.
Toroidal ellipticity need not imply invertibility after restriction to the
holomorphic subspace.  Indeed, the elliptic symbol
\[
 a(x,n)=e^{ix_1}
\]
induces the unilateral shift
\[
 f(z)\longmapsto z_1f(z),
\]
whose range consists of the functions divisible by $z_1$ and hence has
infinite codimension when $d>1$ (and codimension one when $d=1$).  The
invertibility of $L_\lambda$ in
Proposition~\ref{prop:application} instead follows from the explicit
factorization \eqref{eq:application-inverse}.  More generally, a holomorphic
parametrix theory on the polydisk must account for defect spaces which can be
infinite-dimensional when $d>1$; ordinary one-variable Fredholm theory does
not transfer unchanged.
\end{remark}

\end{document}